\documentclass[12pt]{article}
\usepackage[utf8]{inputenc}
\usepackage{graphicx,amssymb,amsmath, mathtools, outlines,caption}
\usepackage{mathrsfs}
\usepackage{tikz}
\usetikzlibrary{shapes,arrows}
\usepackage{comment}
\usepackage[utf8]{inputenc}
\usepackage{amsmath,float}
\usepackage{amssymb,amsthm}
\usepackage{wrapfig}
\usepackage{graphicx}
\usepackage[utf8]{inputenc}
\usepackage{graphicx}
\usepackage{float}
\usepackage[margin=0.75in]{geometry}
\usepackage{tikz}

\newtheorem{theorem}{Theorem}

\newtheorem{prop}{Proposition}
\newtheorem{lemma}[prop]{Lemma}

\newtheorem{example}[prop]{Example}

\theoremstyle{definition}
\newtheorem{definition}[prop]{Definition}

\newtheorem{rmk}{Remark}

\newtheorem{obs}{Observation}

\title{Tracking the Variety of Interleavings}
\author{Ojaswi Acharya, Stella Li, David Meyer and Jasmine Noory }
\date{September, 2020}

\begin{document}

\maketitle

\begin{abstract}
\noindent
In topological data analysis persistence modules are used to distinguish the legitimate topological features of a finite data set from noise.  Interleavings between persistence modules feature prominantly in the analysis.  One can show that for $\epsilon$ positive, the collection of $\epsilon$-interleavings between two persistence modules $M$ and $N$ has the structure of an affine variety,  Thus, the smallest value of $\epsilon$ corresponding to a nonempty variety is the interleaving distance.  With this in mind, it is natural to wonder how this variety changes with the value of $\epsilon$, and what information about $M$ and $N$ can  be seen from just the knowledge of their varieties.  

In this paper, we focus on the special case where $M$ and $N$ are interval modules.  In this situation we classify all possible progressions of varieties, and determine what information about $M$ and $N$ is present in the progression.
\end{abstract}

\section{Acknowledgements}
The research in this paper was conducted by Ojaswi Acharya, Stella Li, Jasmine Noory and David Meyer while Acharya, Li and Noory were students at Smith College in the 2018-19 academic year.

\section{Introduction}
In  \cite{meehan_meyer_1} it was shown that the collection of $\epsilon$-interleavings between two generalized persistence modules for a poset has the structure of an affine variety.  In this paper, we consider the class of (truncated) one-dimensional persistence modules which arise from applying homology to a filtration of a finite data set.  Thus, we confine our attention to persistence modules which are direct sums of (finite) interval modules of the form $[\alpha,\beta)$, for $\alpha < \beta$.  

We will investigate what happens when one studies the {\emph{full collection}} of interleavings between two persistence modules, as opposed to the interleaving distance between them.  If $M, N$ are persistence modules and $\epsilon>0$, the collection of $\epsilon$-interleavings between $M$ and $N$ has the structure of a variety $V_{M,N}^{\epsilon}$.  The interleaving distance $D$ between $M$ and $N$, is the smallest value of $\epsilon$ associated to a nonempty variety (see Figure \ref{figure progression}).  

\begin{figure}[H]
\begin{center}
  \begin{tikzpicture}
 \draw[->]  (0,0) -- (10,0);
 \draw (0,.1) -- (0,-.1) node[below] {\scriptsize{0}};
 \draw (2,.1) -- (2, -.1) node[below,color=blue] {\scriptsize{$D=\epsilon_1$}};
 \draw (5,.1) -- (5, -.1) node[below] {\scriptsize{$\epsilon_2$}};
  \draw (8,.1) -- (8, -.1) node[below] {\scriptsize{$\epsilon_3$}};
  \draw [-,dotted,thick] (10,.5) -- (10.2,.5);
  \draw [-,color=blue] (2,.1) -- (2,.5);
  \draw [-] (5,.1) -- (5,.5);
  \draw [-] (8,.1) -- (8,.5);
 
 \node at (1,1) {\includegraphics[scale=0.08]{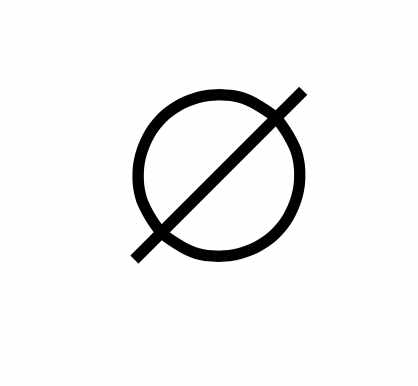}};
 \node[color=blue] at (3,1) {$V_{M,N}^D=V^1$};  
 \node at (6,1) {$V_{M,N}^{\epsilon_2}=V^2$}; 
 \node at (9.1,1) {$V_{M,N}^{\epsilon_3}=V^3$}; 
\end{tikzpicture}
\end{center} 
\caption{The progression of varieties $V_{M,N}^{\epsilon}$ and the interleaving distance \textcolor{blue}{$D$}}
\label{figure progression}
\end{figure}
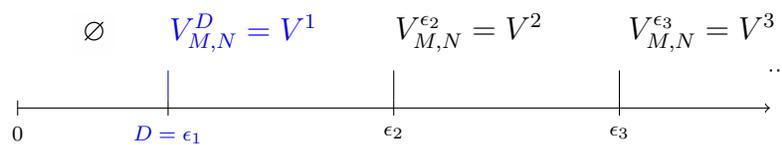

In this setting many questions come to mind.  Which varieties can  appear in this way?  As $\epsilon$ varies, how do they change?  Which ordered lists ({\emph{progressions}}) of varieties are possible?  Does the progression itself detect something about the interleaving distance between the persistence modules that give rise to it?  In our main result, {\bf{Theorem \ref{theorem main}}}, we answer all these questions in the special case that both $M$ and $N$ are interval modules.  

It follows from the completion of the isometry theorem (see \cite{induced_matchings}), that every $\epsilon$-interleaving comes from an $\epsilon$-matching.  Moreover, from \cite{meehan_meyer_2}, \cite{lesnick},  it's clear that the interleaving distance $D$ between the interval modules $[a,b)$ and $[c,d)$ is given by the formula
\begin{equation}
\label{equation distance}
D=min\{m_1,m_2\}\textrm{, where } m_1 = max\{|a-c|,|b-d|\} \textrm{ and }
 m_2 = max\left\{ \tfrac{b-a}{2},\tfrac{d-c}{2} \right\} 
\end{equation}
In {\bf{Proposition} \ref{prop hom}} we give a physical interpretation for these numbers, showing for example that $m_1$  corresponds to the place where the {\emph{last}} homomorphism between the two intervals in born.  The numbers $m_1$ and $m_2$ are then used in {\bf{Theorem} \ref{theorem main}} in our classification of the progression $V_{M,N}^{\epsilon}$, when $M$ and $N$ are both intervals.

This paper is organized as follows.  In {\bf{Section} \ref{section prelim}} we remind the reader of some preliminaries.  In {\bf{Section} \ref{section var}} we define the variety of interleavings between two persistence modules, and then in  {\bf{Section} \ref{section progr}} we give some examples of progressions of varieties.  Lastly, in {\bf{Section} \ref{section main}} we show that our examples in {\bf{Section} \ref{section progr}} constitute an exhaustive list, and we prove our main results.

\section{Preliminaries}
\label{section prelim}
In this Section, we give a brief review of one-dimensional persistence modules.  For a more extensive introduction, see \cite{oudot}.

\begin{definition}
\label{definition pms and morphisms}
A persistence module $M$ is an assignment 
 $$\textrm{ of vector spaces }\{M(x)\} \textrm{, for  }x \in [0,\infty)\textrm{ and  linear maps }\{M(x \leq y)\} \textrm{, for }x \leq y \textrm{ satisfying; } $$
\begin{enumerate}
\item $M(x \leq y):M(x) \to M(y)  \textrm{ and } $
\item $M(x \leq z)=M(y\leq z)\circ M(x \leq y)$ for all $x, y, z$ with $x \leq y \leq z$.
\end{enumerate}
If $M$ and $N$ are two persistence modules, a morphism from $F:M \to N$ is a collection of linear maps $\{F(x):M(x) \to N(x)\}$ indexed by $x \in [0,\infty)$ with the property that 
\begin{equation}
F(y) \circ M(x \leq y)=N(x \leq y)\circ F(x)\textrm{ for } x\leq y. 
\end{equation}
\end{definition}

Thus, a morphism from $M$ to $N$ is a family of linear maps that gives rise to the commutative diagram in Figure \ref{figure morphism}.  One can view persistence modules as functors on a thin category, in which case morphisms correspond to natural transformations.  This categorical perspective will not be necessary for our purposes.

\begin{figure}
\begin{center}
\includegraphics[scale=.9]{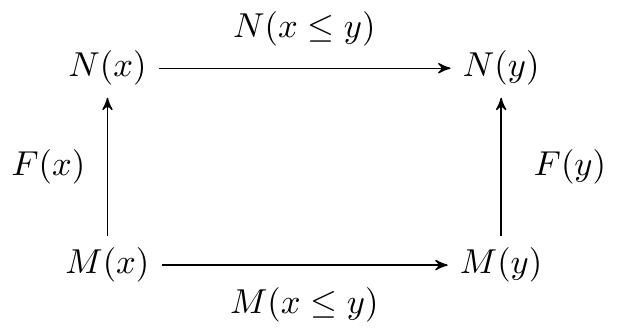}
\end{center}
\caption{$F$ is a morphism from $M$ to $N$}
\label{figure morphism}
\end{figure}

The monoid $({\mathbb{R}}_{\geq 0},+)$ acts on the category of persistence modules on the by the formulas;

\begin{equation}
(M\cdot \epsilon)(x)=M(x+\epsilon)\textrm{, and } (F\cdot \epsilon)(x)=F(x + \epsilon)
\end{equation}
In the latter identity, if $F: M \to N$, then $F \cdot \epsilon : M\cdot \epsilon \to N \cdot \epsilon$.  Thus, the action {\emph{shifts to the left by $\epsilon$}}.  When there is no ambiguity, we suppress the dot and simply write $M\epsilon$ and $F\epsilon$ respectively.  

If  $\alpha < \beta$, the interval module $M_{[\alpha,\beta)}$ is the persistence module defined by the equations;
\begin{equation}
M_{[\alpha,\beta)}(x)=\begin{cases}\mathbb{R}, & x \in [\alpha,\beta)\\ 0, & x \notin [\alpha,\beta)     \end{cases} \hspace{.3 in} M_{[\alpha,\beta)} (x \leq y)=\begin{cases} Id, & x,y \in [\alpha, \beta)\\ 0, & \textrm{ otherwise.} \end{cases}
\end{equation}

Thus, $M_{[\alpha,\beta)}$ is the persistence module whose {\emph{support}} is given by the interval $[\alpha,\beta)$.  Whenever there is no ambiguity, we write $[\alpha,\beta)$ for $M_{[\alpha,\beta)}$.  If $M$ and $N$ are two persistence modules, then $H(M,N)$ is the collection of morphisms from $M$ to $N$.  If $M, N$ are interval modules, $H(M,N)$ is either one-dimensional or identically zero.

\begin{lemma}
\label{lemma hom}
Let $\alpha < \beta$ and $\gamma < \delta$.  Then $H([\alpha,\beta),[\gamma,\delta)) \neq \{0\} \iff \gamma \leq \alpha < \delta \leq \beta$.  
\end{lemma}
This result is well-known, and follows immediately from the requirement that a quotient of $[\alpha,\beta)$ be isomorphic to a submodule of $[\gamma,\delta)$.  One immediately checks that if indeed $\gamma \leq \alpha < \delta \leq \beta$, every morphism $G$ from $[\alpha,\beta)$ to $[\gamma,\delta)$ is given by 
$$\begin{cases}G(x)=\lambda \cdot Id, & x \in [\gamma,\beta)\\ G(x) = 0, & \textrm{ otherwise }\end{cases} $$
for a unique choice of real number $\lambda$.  Thus, when nonzero, the collection of morphisms between interval modules are parametrized by real numbers.  In a slight abuse of notation, we sometimes write 
$$H([\alpha,\beta),[\gamma,\delta))=\lambda 1_{[\gamma,\beta)}$$
where we think of $"1"$ as the characteristic function on the specified interval.
\begin{lemma}
\label{lemma hom interval}
Let $\alpha<\beta, \gamma<\delta$.  Then, 
$$\{x|H([\alpha,\beta),[\gamma,\delta)x) \neq \{0\}\}=\begin{cases} [r,s)&\textrm{ for some }r<s \\ \emptyset .     \end{cases} $$
Moreover, the latter occurs exactly when $\delta \leq \alpha$.
\end{lemma}
That is to say, when $M$ and $N$ are interval modules, $\{x|H(M,Nx) \neq \{0\} \}$ is an interval, which is degenerate only if $M$ and $N$ are disjoint, with $N$ {\emph{below}} $M$.  Moroever, $H([\alpha,\beta),[\alpha,\beta)x)$ is not identically $\{0\}$ on the interval $[0,\frac{\beta -\alpha}{2})$.  The number $W([\alpha,\beta))=\frac{\beta-\alpha}{2}$ is the {\emph{width}} function used in the bottleneck (or Wasserstein) distance on persistence modules.  The set of values for $x$ where $H(M,Nx)$ is not zero will occur in what follows with enough frequency to justify some notation.

\begin{definition}
Let $M, N$ be two persistence modules.  Then let 
$$S_{M,N}=\{x \geq 0| H(M,Nx)\neq \{0\}\}.$$
Thus, {\bf{Lemma} \ref{lemma hom interval}} says that if $M$ and $N$ are intervals modules, then $S_{M,N}$ is either an interval or empty.  
\end{definition}

\section{The Variety of Interleavings}
\label{section var}

\begin{definition}
\label{definition interleaving}
Let $M, N$ be two persistence modules, and say $\epsilon >0.$ An $\epsilon$-interleaving between $M$ and $N$ is a pair of morphism $\Phi:M \to N\epsilon$ and $\Psi:N \to M\epsilon$ satisfying
\begin{eqnarray}
\label{equation interleaving}
\Psi\epsilon \circ \Phi= \Pi_M^{M2\epsilon},\Phi\epsilon \circ \Psi= \Pi_N^{N2\epsilon}
\end{eqnarray}
where for a persistence module $P$, and $\tau >0,$ the morphsim $\Pi_P^{P\tau}: P \to P\tau$ is given by
$$\Pi_P^{P\tau}(x)=P(x \leq x+\tau).$$
\end{definition}

From the definition, it is clear that whether the morphisms $\Phi, \Psi$ constitute an $\epsilon$-interleaving depends only on the triple $(\Phi,\Psi,\epsilon)$ in the sense that the remainder of the interleaving diagram is {\emph{forced}}.  That is, given only $\Phi, \Psi$ and $\epsilon$ one fills in the remainder of the diagram and simply {\emph{checks}} whether the conditions (\ref{equation interleaving}) hold. 

\begin{figure}[h]
\begin{center}
\includegraphics[scale=.8]{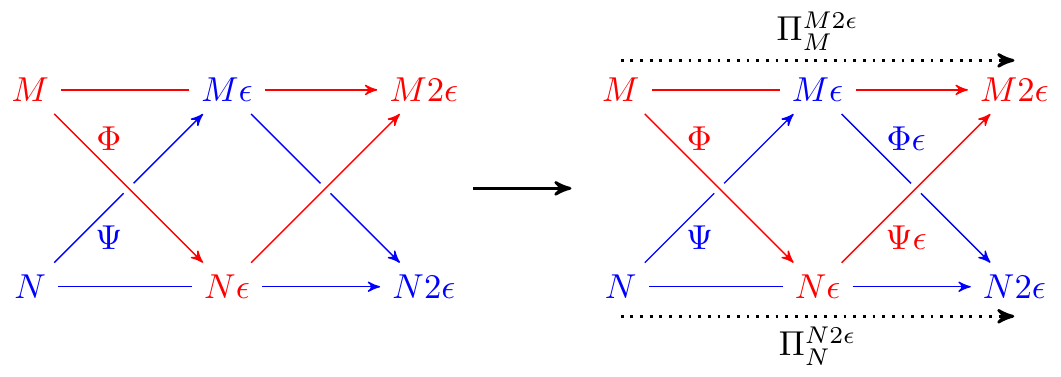}
\caption{An $\epsilon$ interleaving between $M$ and $N$}
\label{figure interleaving}
\end{center}
\end{figure}

$\textrm{If }M=\bigoplus\limits_i^m M_i \textrm{ and } N=\bigoplus\limits_j^n N_j \textrm{ using the Krull-Schmidt property, we may decompose } \Phi \textrm{ and }\Psi.$

$$\textrm{So, }\Phi=(\phi_i^j), \Psi=(\psi_j^i)\textrm{, where,    }\phi_i^j:M_j \to N_i\epsilon\textrm{ and }\psi_j^i:N_i \to M_j\epsilon.$$

If the $M_i$ and $N_j$ are interval modules, by {\bf{Lemma} \ref{lemma hom}}, any of $\phi_i^j, \psi_j^i$ which could be nonzero are given by a real parameter.  Moreover, the value of the parameter is preserved under the action of ${\mathbb{R}}_{\geq 0}$ when one shifts by $\epsilon$.  These will be the {\emph{variables}} used for the coordinate ring for our variety.  

We now define the variety.  Let $K$ and $L$ be given by 
$$K \in M_{n,m}(\mathbb{R}), K=( k_a^b)\textrm{ and }L\in M_{m,n}(\mathbb{R}), L=({\ell}_a^b) .$$

$K$ and $L$ will correspond to the decomposition of the morphisms $\Phi$ and  $\Psi$ respectively.  Thus, we identify the $k_a^b$ with ${\phi}_a^b$, and ${\ell}_a^b$ with ${\psi}_a^b$.  

Clearly the matrix for $\Pi_M^{M2\epsilon}$ is diagonal with the morphism $\Pi_{M_i}^{M_i2\epsilon}$ constituting the $(i,i)$-entry.  Moreover, it's easy to see that if  $M_i=[\alpha_i,\beta_i)$, 
$$\Pi_{M_i}^{M_i2\epsilon}= \begin{cases} 1_{[\alpha_i,\beta_i-2\epsilon)}, & \epsilon < \tfrac{\beta_i-\alpha_i}{2}\\ 0, & \textrm{ otherwise } \end{cases}$$
\noindent
Thus, we have $\Pi_{M_i}^{M_i2\epsilon} \neq 0 \iff H(M_i,M_i2\epsilon) \neq \{0\}\iff \epsilon < \tfrac{\beta_i-\alpha_i}{2} \iff 2\epsilon \in S_{M_i,M_i},$ with the analogous statement holding for the diagonal in $\Pi_N^{N2\epsilon}$.  Therefore, substituting into (\ref{equation interleaving}), our commutativity conditions become
\begin{eqnarray}
\label{equation variety}
L \cdot K = \Pi_{M}^{M2\epsilon}, K\cdot L=\Pi_{N}^{N2\epsilon},
\end{eqnarray}
subject to two qualifications.  First, many of the entries of $K$ and $L$ must be zero simply because there are no such nonzero morphisms.  Secondly, while there are $m^2+n^2$ equations involving the variables in (\ref{equation variety}), only those corresponding to entries which {\emph{could}} vary must be satisfied.  For example, the equations
$$\sum\limits_b^n k_5^b \ell_b^1=0 \textrm{ or } \sum\limits_b^n \ell_3^b k_b^3=1$$
are only constraints on those variables which appear {\emph{only}} when the sets $H(N_1,N_5 2\epsilon)$ and $H(M_3,M_32\epsilon)$ are not $\{0\}$ respectively.  With this in mind, let  
$$
\mathcal{S}=\{X_{(a,b)} | 2\epsilon \in S_{M_b,M_a}\} \cup \{ Y_{(a,b)} | 2\epsilon \in S_{N_b,N_a}\},
\mathcal{T}=\{k^a_b | \epsilon \notin S_{M_a,N_b} \}\cup \{\ell^a_b | \epsilon \notin S_{N_a,M_b}\},$$
where $X=L\cdot K-\Pi_{M}^{M2\epsilon}$, $Y=K\cdot L-\Pi_{N}^{N2\epsilon}$, and the subscript corresponds to the $(a,b)$-entry of the matrix.  Thus, $\mathcal{S}$ is the collection of matrix conditions that the variables must satisfy, and $\mathcal{T}$ corresponds to those variables which should be regarded as {\emph{missing}} from the matrices $K, L$, since they can only be zero.

Finally, we are ready to define the variety at $\epsilon$.
\begin{definition}
\label{def variety}
Using the above notation, the variety of $\epsilon$-interleavings between $M$ and $N, V_{M,N}^\epsilon$ is the affine variety whose coordinate ring given by 
$$A_{M,N}^\epsilon= \mathbb{R}[\{k_b^a\, \ell_a^n | 1 \leq a \leq m, 1 \leq b \leq n\}]\big/(\mathcal{S}\cup \mathcal{T}).$$
\end{definition}

We will now illustrate with some examples.

\begin{example}
\label{example variety 1}
Let $M=M_1 \oplus M_2, N=N_1 \oplus N_2$, where $M_1=N_1=[1,4), M_2=[1.2,3.9)$ and $N_2=[.9,4.1)$.

By the isometry theorem, the interleaving distance is $.2$ corresponding to the $.2$-matching
$$M_1 \leftrightarrow N_2 \textrm{ and }M_2 \leftrightarrow N_1. $$
We will compute the variety at the interleaving distance $\epsilon=.2$.  First, by inspection
\begin{eqnarray}
S_{M_1,N_1}=[0,3), S_{M_1,N_2}=[.1,3.1), S_{M_2,N_1}=[.1,2.8) \textrm{ and }S_{M_2,N_2}=[.1.2.9).\\ 
\textrm{Also, }S_{N_1,M_1}=[0,3), S_{N_1,M_2}=[.2,2.9), S_{N_2,M_1}=[.1,3.1) \textrm{ and }S_{N_2,M_2}=[.3,3).
\end{eqnarray}
Since only $S_{N_2,M_2}$ does not contain $\epsilon=.2$ as an element,
$$K=\begin{pmatrix}k^1_1 & k^2_1\\ k^1_2 & k^2_2\end{pmatrix}, L=\begin{pmatrix}\ell^1_1 & \ell^2_1\\ \ell^1_2 & 0\end{pmatrix}
\textrm{ and }\mathcal{T}=\{ \ell^2_2\}.$$

Moreover,
\begin{eqnarray}
S_{M_1,M_1}=[0,3), S_{M_1,M_2}=[.2,2.9), S_{M_2,M_1}=[.1,2.8),S_{M_2,M_2}=[0,2.7),\\
S_{N_1,N_1}=[0,3), S_{N_1,N_2}=[.1,3.1), S_{N_2,N_1}=[.1,3.1)\textrm{ and }S_{M_2,N_2}=[0,3.2).
\end{eqnarray}
Thus, since $.4=2\epsilon$ is in all these intervals, all matrix equations must be satisfied, and 
 $$\Pi_M^{M.4}=\begin{pmatrix}1 & 0\\0 &1 \end{pmatrix}= \Pi_N^{N.4}.$$
One can check that the variety $V_{M,N}^{.2}$ therefore corresponds to the polynomial equations
$$\ell^2_2, k^1_1, \ell^2_1 k^1_2-1, \ell^1_2 k^2_1-1, \ell^1_1 k^2_1 + \ell^2_1 k^2_2, \ell^1_1 k^1_2 + \ell^1_2 k^2_2$$
\end{example}

\begin{example}
\label{example variety 2}
We now consider the same peristence modules as in {\bf{Example} \ref{example variety 1}}, now with a different value of the parameter $\epsilon$.  This time, by $(3)$ and $(4)$, $\mathcal{T}=\emptyset$, since $\epsilon=.4$ is in all the intervals.  Also, since $2\epsilon$ is in all the intervals in $(5), (6)$, all matrix equations must be satisfied and
$$K=\begin{pmatrix}k^1_1 & k^2_1\\ k^1_2 & k^2_2\end{pmatrix},  L=\begin{pmatrix}\ell^1_1 & \ell^2_1\\ \ell^1_2 & \ell^2_2\end{pmatrix}, \Pi_M^{M.4}=\begin{pmatrix}1 & 0\\0 &1 \end{pmatrix}= \Pi_N^{N.4}.$$
Thus, our variety of interleavings $V_{M,N}^{.5}$ is now given by the equations
$$K \cdot L = \begin{pmatrix}1 & 0\\ 0 &1\end{pmatrix} = L \cdot K $$
Thus, we see that the $.4$-interleavings are parametrized by the set
$$\{(Z,Z^{-1})| Z \in Gl_2(\mathbb{R})\} \subseteq {\mathbb{R}}^8.$$
Note that all $.2$-interleavings remain $.4$-interleavings, though from $.2$ to $.4$, the variety grows in dimension.
\end{example}

\begin{example}
\label{example variety 3}
We continue, considering the same pair of persistence modules.  Now, let $\epsilon=1.55$.  As was {\bf{Example}\ref{example variety 2}}, since $\epsilon=1.55$ is in all the intervals in $(3)$ and $(4)$, $\mathcal{T}=\emptyset$.  Now, however, from $(5), (6)$, we see that $2\epsilon=3.1$ is only in $S_{N_2,N_2}$.  Thus, we see that $V_{M,N}^{1.55}$ is given by
$$k^1_2 \ell^2_2 + k^ 2_2\ell^2_2-1$$
where all other variables are free.
\end{example}

\begin{example}
\label{example variety 4}
One last example.  Now let $\epsilon=3$.  Here the variety $V_{M,N}^3$ is given by equations
$$k^1_1, k^2_1, k^2_2, \ell^1_1, \ell^1_2, \ell^2_2 $$
\end{example}
where $k^1_2, \ell^2_1$ are free.  Thus, $V_{M,N}^3$ is a plane in $8$-dimensional affine space.

\section{Some Progressions}
\label{section progr}
As we see in {\bf{Examples} \ref{example variety 1}-{\ref{example variety 4}}}, the interleavings between a fixed pair of persistence modules gives rise to a progression of varieties indexed by the parameter $\epsilon$.  It's clear that for a fixed $M$ and $N$ only finitely many different varieties actually appear in the range of the assignment
$$\epsilon \to V_{M,N}^{\epsilon}. $$
In this section, we provide examples showing how the variety of $\epsilon$-interleavings associated to two interval modules changes with the value of the parameter $\epsilon$.  When $M$ and $N$ are clear from the context, we'll write $V^i$ for the $i$th nonempty variety in the progression, suppresing $\epsilon$.  Note that since $M$ and $N$ are intervals, the matrices $K, L, \Pi_M^{M2\epsilon}$ and $\Pi_N^{N2\epsilon}$ are all scalar matrices.  Now, we'll provide some examples of progressions.

\begin{example}
\label{origin1}

$V^1$ is the origin , $V^2$ is a coordinate axis, and $V^3$ is the origin.

\begin{center}
\begin{tikzpicture}
 \draw[->]  (0,0) -- (10,0);
 \draw (0,.1) -- (0,-.1) node[below] {\scriptsize{0}};
 \draw (2,.1) -- (2, -.1) node[below] {\scriptsize{1}};
 \draw (5,.1) -- (5, -.1) node[below] {\scriptsize{6}};
  \draw (8,.1) -- (8, -.1) node[below] {\scriptsize{7}};
 
 \node at (1,1) {\includegraphics[scale=0.08]{null.png}};
 \node at (3.5,1) {\includegraphics[scale=0.08]{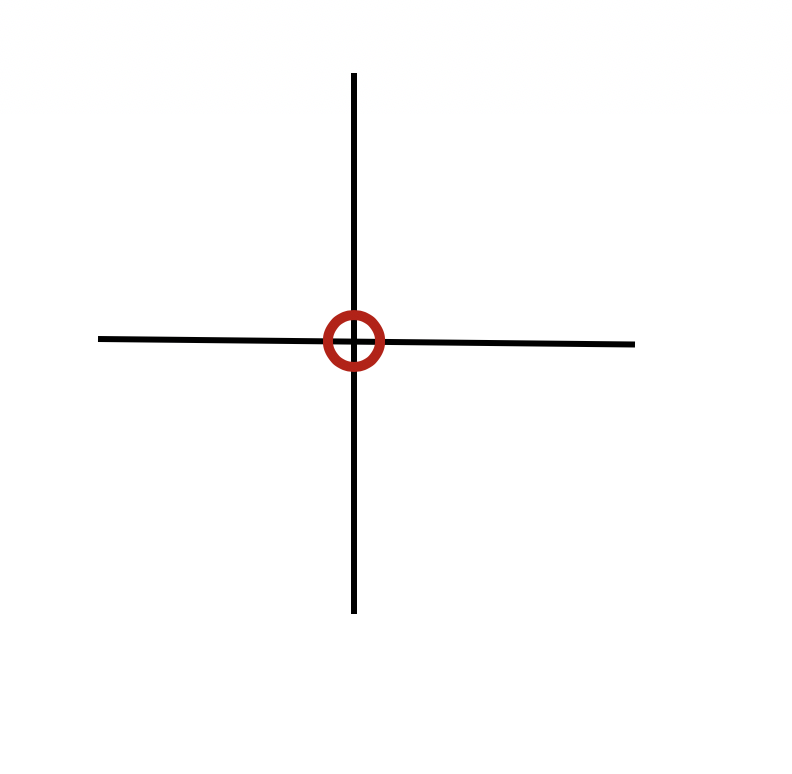}};  
 \node at (6.25,1) {\includegraphics[scale=0.08]{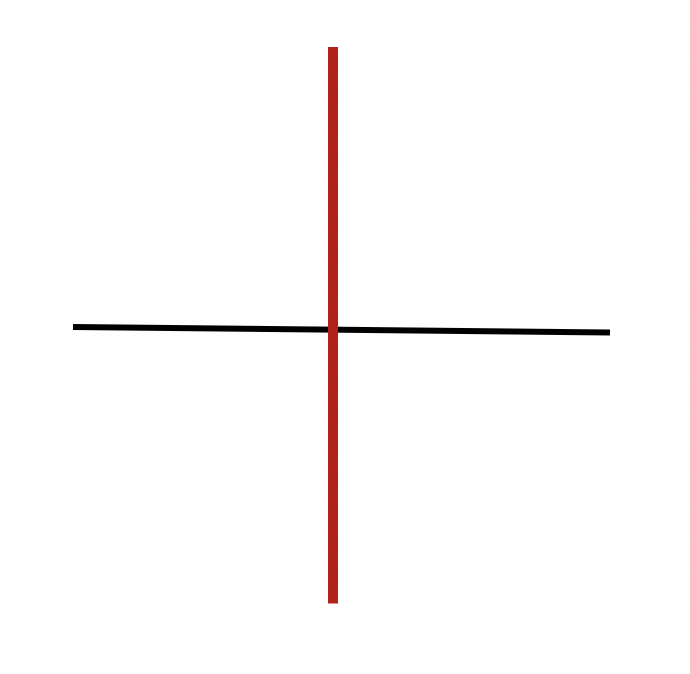}}; 
 \node at (8.25,1) {\includegraphics[scale=0.08]{origin.png}}; 
\end{tikzpicture}
\end{center}

\begin{center}
and
\vspace{.2in}

  \begin{tikzpicture}
 \draw[->]  (0,0) -- (10,0);
 \draw (0,.1) -- (0,-.1) node[below] {\scriptsize{0}};
 \draw (2,.1) -- (2, -.1) node[below] {\scriptsize{1}};
 \draw (5,.1) -- (5, -.1) node[below] {\scriptsize{6}};
  \draw (8,.1) -- (8, -.1) node[below] {\scriptsize{7}};
 
 \node at (1,1) {\includegraphics[scale=0.08]{null.png}};
 \node at (3.5,1) {\includegraphics[scale=0.08]{origin.png}};  
 \node at (6.25,1) {\includegraphics[scale=0.08]{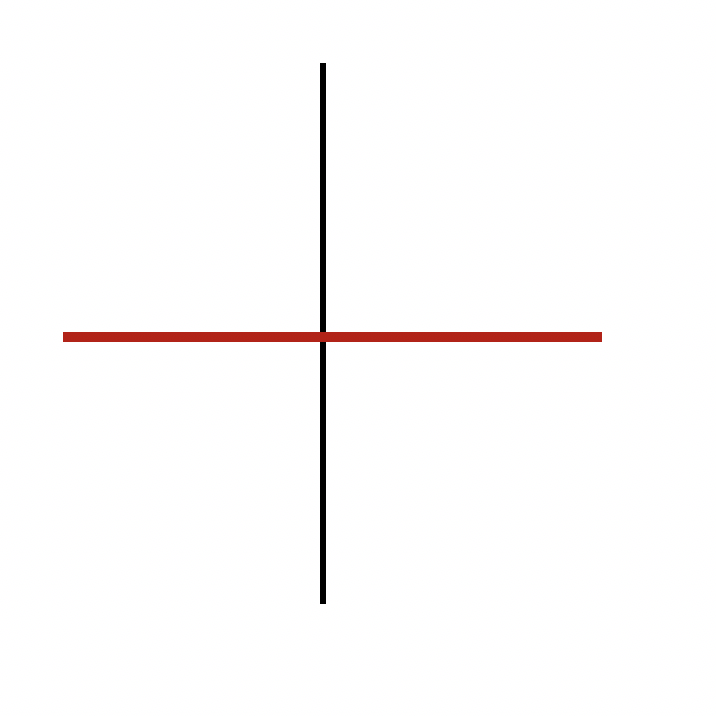}}; 
 \node at (8.25,1) {\includegraphics[scale=0.08]{origin.png}}; 
\end{tikzpicture}
\end{center}


That is, the complete progression is origin, axis, orgin.  First, let $M = [6,8)$ and $N = [1,2).$  Note that 
$$S_{M,N}=\emptyset \textrm{, and } S_{N,M}=[6,7), S_{M,M}=[0,2), S_{N,N}=[0,1).$$
Recall that the last two intervals tell us that
 $$H(M,M2\epsilon)=\begin{cases} \mathbb{R} & \epsilon \in [0,1) \\ \{0\} & \epsilon \in [1,\infty)\end{cases} \textrm{ and }\hspace{.1 in}  H(N,N2\epsilon)=\begin{cases} \mathbb{R} & \epsilon \in [0,\tfrac{1}{2}) \\ \{0\} & \epsilon \in [\tfrac{1}{2},\infty)\end{cases} $$

Now, for $ 0\leq \epsilon < 1$, 
$$K=(0), L=(0), \Pi_M^{M2\epsilon}=(1), \mathcal{T}=\{k,\ell \}. $$
Since $2\epsilon \in S_{M,M}$, our matrix condition must be satisfied.  Thus, the variety is given by
$$k,\ell, k\ell-1, $$
so we have the empty variety.  Note that this is exactly because $\epsilon$ is less than the interleaving distance of 1.

Next, for $1 \leq \epsilon < 6$\\
       $$K=(0), L=(0), \Pi_M^{M2\epsilon}=(0), \Pi_N^{N2\epsilon}=(0), \mathcal{T}=\{k,\ell\}, $$
thus no matrix conditions need to be satisfied, and $V^1$ is 
$$k, \ell $$
That is, $V^1$ is the origin in ${\mathbb{R}}^2$.

 For $6 \leq \epsilon \leq 7$\\
     $$K=(0), L=(\ell), \Pi_M^{M2\epsilon}=(0), \Pi_N^{N2\epsilon}=(0), \mathcal{T}=\{k\}, $$
 so no matrix conditions need to be satisfied, and $V^2$ is given by
 $$k $$
 so $k=0$ and $\ell$ is free.  Thus, $V^2$ is the $y$-axis.

 Lastly, for $\epsilon \geq 7$, once again\\
    $$K=(0), L=(0), \Pi_M^{M2\epsilon}=(0), \Pi_N^{N2\epsilon}=(0), \mathcal{T}=\{k,\ell\}.$$

Thus, $V^3$ is the origin again.  Of course, it's clear that by switching the roles of $M, N$ we may obtain the progression origin, x-axis, origin.\\\\

\end{example}

The next Examples are provided without justification.  The reader can easily check that each progression matches the given pair of interval modules.  
  
\begin{example}
\label{axis1}

$V^1$ is a coordinate axis, $V^2$ is the origin.\\

    \begin{center}
\begin{tikzpicture}
 \draw[->]  (0,0) -- (10,0);
 \draw (0,.1) -- (0,-.1) node[below] {\scriptsize{0}};
 \draw (3,.1) -- (3, -.1) node[below] {\scriptsize{1}};
  \draw (7,.1) -- (7, -.1) node[below] {\scriptsize{3}};
 
 \node at (1.5,1) {\includegraphics[scale=0.08]{null.png}};
 \node at (5,1) {\includegraphics[scale=0.08]{yaxis.png}};  
 \node at (8,1) {\includegraphics[scale=0.08]{origin.png}}; 
\end{tikzpicture}
\end{center}

 \begin{center}
 and 
 
 \vspace{.2in}
\begin{tikzpicture}
 \draw[->]  (0,0) -- (10,0);
 \draw (0,.1) -- (0,-.1) node[below] {\scriptsize{0}};
 \draw (3,.1) -- (3, -.1) node[below] {\scriptsize{1}};
  \draw (7,.1) -- (7, -.1) node[below] {\scriptsize{3}};
 
 \node at (1.5,1) {\includegraphics[scale=0.08]{null.png}};
 \node at (5,1) {\includegraphics[scale=0.08]{xaxis.png}};  
 \node at (8,1) {\includegraphics[scale=0.08]{origin.png}}; 
\end{tikzpicture}
\end{center} 

\vspace{1cm}

Let $M = [1,3), N = [0,2)$ and $M=[0,2), N=[1,3)$ respectively.

\end{example}

\begin{example}
\label{hyperbola1}


$V^1$ is a hyperbola, $V^2$ is a plane, $V^3$ is a coordinate axis and $V^4$ is the origin.\\

\begin{center}
\begin{tikzpicture}
 \draw[->]  (0,0) -- (10,0);
 \draw (0,.1) -- (0,-.1) node[below] {\scriptsize{0}};
 \draw (2,.1) -- (2, -.1) node[below] {\scriptsize{.2}};
 \draw (4,.1) -- (4, -.1) node[below] {\scriptsize{.7}};
  \draw (6,.1) -- (6, -.1) node[below] {\scriptsize{1.2}};
    \draw (8,.1) -- (8, -.1) node[below] {\scriptsize{1.3}};
 
 \node at (1,1) {\includegraphics[scale=0.08]{null.png}};
 \node at (3,1) {\includegraphics[scale=0.08]{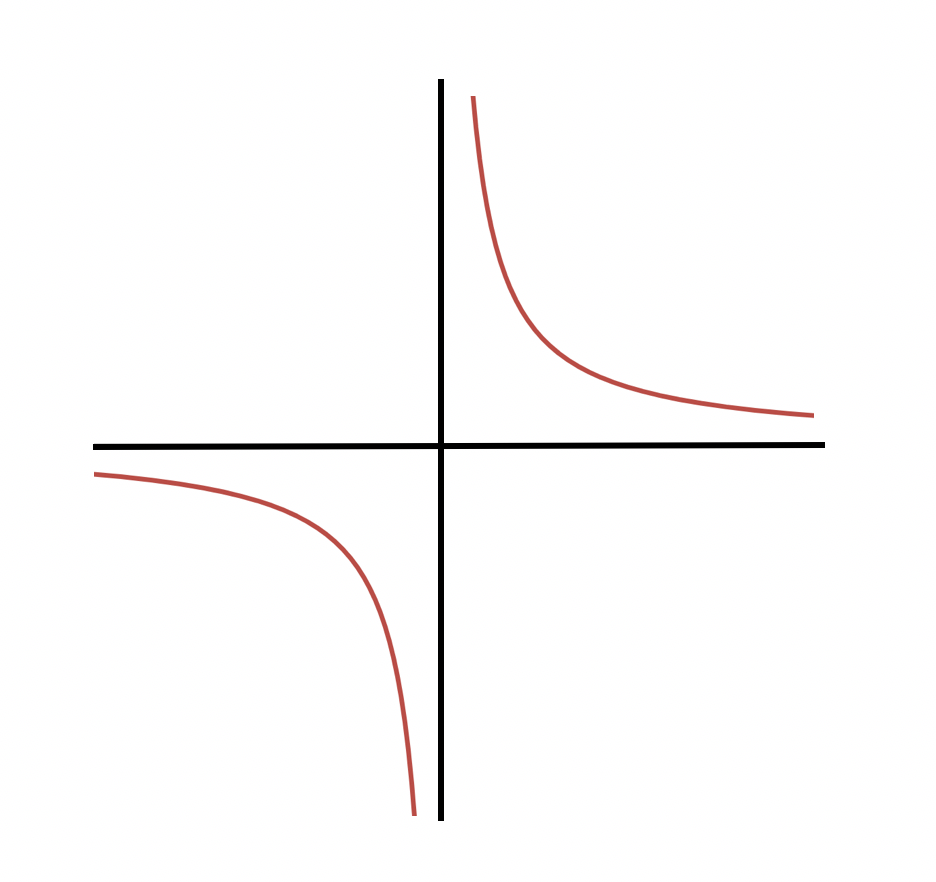}};  
 \node at (5,1) {\includegraphics[scale=0.08]{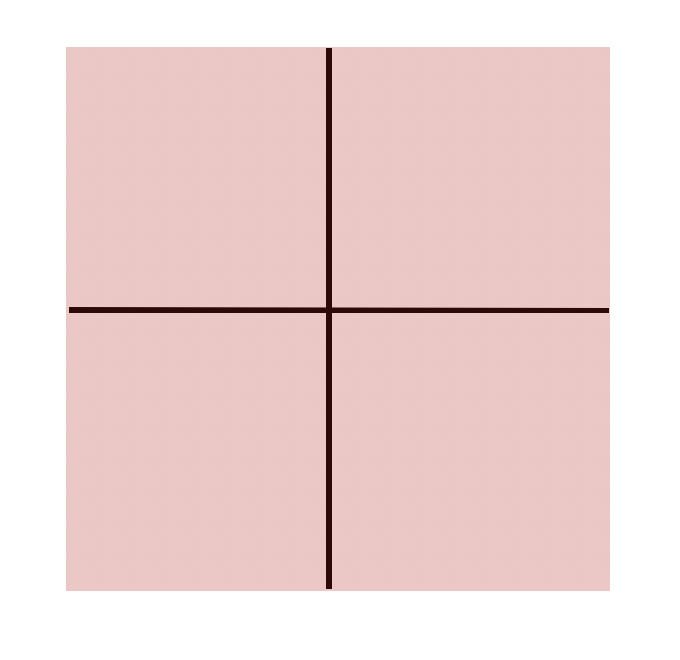}}; 
  \node at (7,1) {\includegraphics[scale=0.08]{yaxis.png}}; 
 \node at (9,1) {\includegraphics[scale=0.08]{origin.png}}; 
\end{tikzpicture}
\end{center}

\begin{center}
and

\vspace{.2in}

\begin{tikzpicture}
 \draw[->]  (0,0) -- (10,0);
 \draw (0,.1) -- (0,-.1) node[below] {\scriptsize{0}};
 \draw (2,.1) -- (2, -.1) node[below] {\scriptsize{.2}};
 \draw (4,.1) -- (4, -.1) node[below] {\scriptsize{1.1}};
  \draw (6,.1) -- (6, -.1) node[below] {\scriptsize{1.2}};
    \draw (8,.1) -- (8, -.1) node[below] {\scriptsize{1.3}};
 
 \node at (1,1) {\includegraphics[scale=0.08]{null.png}};
 \node at (3,1) {\includegraphics[scale=0.08]{1_x.png}};  
 \node at (5,1) {\includegraphics[scale=0.08]{R2.png}}; 
  \node at (7,1) {\includegraphics[scale=0.08]{xaxis.png}}; 
 \node at (9,1) {\includegraphics[scale=0.08]{origin.png}}; 
\end{tikzpicture}
\end{center}

Let $M = [1,2.1), N = [.8,2.2)$ and $M=[.8,2.2), N=[1,2.1)$ respectively.

\end{example}

\begin{example}
\label{hyperbola2}


$V^1$ is a hyperbola, $V^2$ is the plane and $V^3$ is the origin.\\

Let M = [.9,2.1), N = [1,2)

    \begin{center}
\begin{tikzpicture}
 \draw[->]  (0,0) -- (10,0);
 \draw (0,.1) -- (0,-.1) node[below] {\scriptsize{0}};
 \draw (2,.1) -- (2, -.1) node[below] {\scriptsize{.1}};
 \draw (5,.1) -- (5, -.1) node[below] {\scriptsize{.5}};
  \draw (8,.1) -- (8, -.1) node[below] {\scriptsize{.9}};
 
 \node at (1,1) {\includegraphics[scale=0.08]{null.png}};
 \node at (3.5,1) {\includegraphics[scale=0.08]{1_x.png}};  
 \node at (6.25,1) {\includegraphics[scale=0.08]{R2.png}}; 
 \node at (8.5,1) {\includegraphics[scale=0.08]{origin.png}}; 
\end{tikzpicture}
\end{center} 

\vspace{1cm}

\end{example}

In the next Section we will show that {\bf{Examples \ref{origin1},\ref{axis1},\ref{hyperbola1}}} and {\bf{\ref{hyperbola2}}} constitute a complete list of all possible progressions between two interval modules.  We will also connect the first variety $V^1$ to $m_1$ and $m_2$ (see (\ref{equation distance})), the two terms whose minimum is the interleaving distance.  We now prove a useful Proposition which provides a physical interpretation to $m_1$.
 
\begin{prop}
\label{prop hom}
Let $M=[a,b)$ and $N=[c,d)$ and set 

$$\sigma=\left\{
                \begin{array}{ll}
                  min(S_{M,N}), \text{if} \hspace{1mm} S_{M,N} \neq \emptyset\\
                  0, \hspace{16mm} \text{if} \hspace{1mm}S_{M,N} = \emptyset
                \end{array}
              \right.
\tau=\left\{
                \begin{array}{ll}
                  min(S_{N,M}), \text{if} \hspace{1mm} S_{N,M} \neq \emptyset\\
                  0, \hspace{16mm} \text{if} \hspace{1mm}S_{N,M} = \emptyset
                \end{array}
              \right.
$$

and similarly
$$\sigma'=\left\{
                \begin{array}{ll}
                  sup(S_{M,N}), \text{if} \hspace{1mm} S_{M,N} \neq \emptyset\\
                  0, \hspace{16mm} \text{if} \hspace{1mm}S_{M,N} = \emptyset
                \end{array}
              \right.
              \tau'=\left\{
                \begin{array}{ll}
                  sup(S_{N,M}), \text{if} \hspace{1mm} S_{N,M} \neq \emptyset\\
                  0, \hspace{16mm} \text{if} \hspace{1mm}S_{N,M} = \emptyset
                \end{array}
              \right.
_.$$

\noindent
Thus, $Hom(N,Mx)\neq 0 \iff \sigma \leqslant x < \sigma'$ and $S_{N,M}\neq \emptyset$ and similarly $Hom(M,Nx)\neq 0 \iff \tau \leqslant x < \tau'$ and $S_{M,N}\neq \emptyset$.
\newline
\noindent 
Then,
\begin{enumerate}
\item[i.] $max\{|a-c|,|b-d|\}=max \{\sigma, \tau\}$.  So $max\{|a-c|,|b-d|\}$ corresponds to the number where the last homomorphism between $M$ and $N$ is born.
\item[ii.]  $max\{|a-d|,|b-c|\}=max \{\sigma', \tau'\}$.  So ${max}\{|a-d|,|b-c|\}$ corresponds to the number where the last homomorphism between $M$ and $N$ dies.
\end{enumerate}
\end{prop}
\vspace{5mm}
\noindent
\begin{proof}
We proceed in three main cases, conditioning on the way the supports of $M$ and $N$ compare.  In the proof, we use {\bf{Lemmas} \ref{lemma hom},\ref{lemma hom interval}} repeatedly without explicit reference.

{\bf{Case 1:}} The supports are disjoint.  That is, $M \cap N =\emptyset$.  

Without loss of generality, suppose that $a<b \leq c <d$.  Then, $S_{N,M} =  \emptyset$, so  $\tau, \tau'$ are both identically zero.  Therefore, $max\{\sigma, \tau\}=\sigma$ and $max\{\sigma', \tau'\}=\sigma'$.  Thus, we will show that
$$\sigma=max\{|a-c|,|b-d|\}=max\{c-a,d-b\} \textrm{ and }\sigma'=max\{|a-d|,|b-c|\}=d-a.$$

First, say $c-a\geq d-b$.  If $x=c-a$, then $x \in S_{M,N}$, since $H(M,Nx)=H([a,b),[a,d-c+a)\neq \{0\}$, as
 $$a \leq a < d-c+a \leq b.$$ 

Thus, $\sigma \leq c-a$.  But if $\sigma < c-a $, then for some  $x <c-a$, $H([a,b),[c-x,d-x)) \neq \{0\}$.  If this were the case, then  $c-x \leq a$, a contradiction.  Thus we have shown that $\sigma = c-a$.

Now set $x=d-a$.  Then $H([a,b),[c-x,d-x))=H([a,b),[c-d+a,d-d+a))=\{0\}$, since $a \nless a$.  Thus, $\sigma' \leq d-a$.  However, if $\sigma' < d-a $, then there exists $x\in (c-a,d-a)$ with $H([a,b),[c-x,d-x))=\{0\}$.  However, for any such $x, b+x \geq b+c-a \geq b+d-b=d$.  Therefore, 
$$c-x\leq a<d-x\leq b$$ 
Thus $i. ii.$ are proven if $c-a \geq d-b$.

Now, say $d-b > c-a$, and let $x=d-b$.  Then $x \in S_{M,N}$ since 
$$c-x \leq c-c+a=a<b=d-x\leq b.$$  
Thus $\sigma \leq d-b$.  But if $\sigma < d-b$, then there exists $x < d-b$ with $x \in S_{M,N}$  This is a contradiction since $b<d-x$.  Therefore $\sigma = d-b$.  

Moreover, $H([a,b), [c-d+a,a))=\{0\}$, since $a \nless a$ so $\sigma' < d-a$.  However for $x$ with $c-a\leq d-b < x < d-a$, $H([a,b),[c-x,d-x))=\{0\}$, since $c-x \leq a < d-x \leq b$.  Thus, $\sigma'=d-a$.  Our result is therefore established in {\bf{Case 1}}.

\medskip 

\noindent 
{\bf{Case 2:}} The supports intersect, but neither one contains the other.  Explicitly, $M \cap N \neq \emptyset, M \nsubseteq N$ and $N \nsubseteq M$.

Without loss of generality, suppose $a \leq c < b \leq d$.  In this situation, $0 \in S_{N,M}$, so $\tau = 0$ and $S_{M,N} \neq \emptyset$.  Thus, to prove $i.$ we must show that $max\{c-a,d-b\}=\sigma$.  First, suppose that $c-a \geq d-b$.  Then, $H([a,b),[c-c+a,d-c+a))=H([a,b),[a,a+d-c)) \neq \{0\}$ since 
$$a\leq a <a+d-c \leq b.$$
This says that $\sigma \leq c-a$.  On the other hand, if $\sigma < c-a$, then for some  $x < c-a, x \in S_{M,N}$.  If this were the case, then 
$$c-x \leq a <d-x \leq b,$$
a contradiction, since $c-x \geq a$.  Thus, $i.$ is established, when $c-a \geq d-b.$ 

Similarly, say $d-b \geq c-a$.  Then, $H([a,b),[c-d+b,d-d+b))=H([a,b),[b+c-d,b)) \neq \{0\}$, since 
$$b+c-d \leq a <b \leq b.$$
Thus we have that $\sigma \leq d-b$.  If $\sigma < d-b$, then for some $x <d-b, H([a,b),c-x,d-x)) \neq \{0\}$.   If this were true, then we must have 
$$c-x \leq a < d-x \leq b,$$ 
a contradiction, since $x \geq d-b.$  Thus, $\sigma = max\{c-a,d-b\}$, so $i.$ holds.

To establish $ii.$, we'll show that $\tau' = b-c$ and $\sigma'=d-a$.  First, $H([c,d),[a-b+c,b-b+c))=H([c,d),[a-b+c,c))=\{0\}$, since $c \nless c$, thus $\tau' \leq b-c$.  But if $\tau' < b-c$, then for some $x <b-c$ we have $H([c,d),[a-x,b-x))=\{0\}$.  However, for any such $x$,  
$$a-x < c<b-x\leq d,$$
a contradction, since $x \in S_{N,M}$.  Thus $\tau'=b-c$.  Lastly $H([a,b),[c-d+a,d-d+a))=H([a,b),[c-d+a,a))=\{0\}$, since $a \nless a$.  On the other hand, if $max\{c-a,d-b\} \leq x <d-a$, $H([a,b),[c-x,d-x)) \neq \{0\}$, since 
$$c-x \leq a <d-x \leq b.$$ 
Therefore, $\sigma'=d-a$ and $i. ii.$ are proven in {\bf{Case 2}}.

\medskip

\noindent 
{\bf{Case 3:}} One support contains the other.  That is $N\subseteq M$ or $M \subseteq N$.

Without loss of generality, say $M \supseteq N$ and $a \leqslant c < d \leqslant b$.  First say $c-a \leq b-d$.  Then, $c-a+d-c \leq b-d+d-c\implies d-a \leq b-c$.  We'll show 
$$S_{M,N}=[c-a,d-c) \textrm{, and }S_{N,M}= [b-d,b-c).$$
Note that $H([a,b),[c-c+a,d-c+a))=H([a,b),[a,a-d+c)) \neq \{0\}$, since 
$$a \leq a <a+d-c\leq b.$$
Also, for $x<c-a, H([a,b),[c-x,d-x))=\{0\}$, since $c-x>a$.  Thus $\sigma =c-a$.  Moreover, $H([a,b),[c-d+a,d-d+a))=\{0\}$, since $a \nless a$.  But for $x \in (c-a,d-a)$, $H([a,b),[c-x,d-x)) \neq \{0\}$ since 
$$c-x\leq a <d-x<d-c+a<a+d-c+c-a=d<b,$$ 
so $S_{M,N}=[c-a,d-a)$.

In addition, $H([c,d),[a-b+d,b-b+d))=H([c,d),[a-b+d,d))\neq \{0\}$, since 
$$a-b+d\leq d+c-d=c<d\leq d.$$
If $x <b-d$, then $H([c,d),[a-x,b-x))=\{0\}$, since $b-x >d$.  Thus, $\tau = b-d$.  For $x=b-c$ we have $H([c,d),[a-b+c,b-c+c))=H([c,d),[a-b+c,c))=\{0\}$, since $c \nless c$.  However, for $x\in [b-d,b-c)$, $H([c,d),[a-x,b-x)) \neq \{0\}$ since 
$$a-x\leq a-b+d \leq d+c-d = c <b-x \leq d.$$  Thus, $\tau'=b-c$ and $S_{N,M}=[b-d,b-c)$ as required.

The case where $b-d <c-a$ is similar.
\end{proof}

We point out that the proof also shows us that the last homomorphism {\emph{born}} is the last to {\emph{die.}}  When $M=[a,b)$ and $N=[c,d)$, from (\ref{equation distance}) we know that the interleaving distance $D$ between $M$ and $N$ is given by
\begin{equation*}
D= min \left\{ max\{|a-c|,|b-d|\}, max\left\{ \tfrac{b-a}{2},\tfrac{d-c}{2} \right\} \right\}={min\{m_1,m_2\}}_. 
 \end{equation*}
Thus, {\bf{Proposition} \ref{prop hom}} and {\bf{Lemma} \ref{lemma hom interval}} give us another way of interpreting the terms appearing in this formula for the interleaving distance, since $m_1$ is where the last homomorphism is born, and $m_2$ is the smallest number such that both ${\Pi}_M^{M2m_2}$ and ${\Pi}_N^{N2m_2}$ are identically zero.
 
 \section{Main Results}
 \label{section main}
 In this Section we show that the Examples in Section \ref{section progr} constitute an exaustive list of all possible progressions.  We will also show that the first variety in a progression $V^1$, detects which whether the interleaving distance comes from a matching between the intervals $M$ and $N$.
 
 \begin{lemma}
 \label{lemma dead hom}
As above, let $D$ be the interleaving distance between $M$ and $N$ and again, let
$$ m_1 = max\{|a-c|,|b-d|\} \textrm{, }
 m_2 = max\left\{ \tfrac{b-a}{2},\tfrac{d-c}{2} \right\} $$  
If $m_1 \geq m_2$, then only one homomorphism will appear.  That is, exactly one of $S_{M,N}, S_{N,M}$ intersects $[D,\infty)$.
 \end{lemma}
 
 \begin{proof}
 First, the conclusion is clear whenever $M\cap N=\emptyset$ independent of the relationship between $m_1$ and $m_2$, since in this situation exactly one of $S_{M,N}, S_{N,M}$ is not empty.  In fact, if the supports of $M$ and $N$ are disjoint, then necessarily $m_1 > m_2$.  Now, say $m_1 > m_2$.  
 
 First, suppose the support of $M$ and $N$ intersect, but neither contains the other.  Without loss of generality, say $a < c <b <d$.  Then, 
 $$S_{N,M}=[0,b-c) \textrm{, and }S_{M,N}=[m_1,d-a). $$
 We must show that $D=m_2 \geq b-c$.  For a contradiction, suppose that $\frac{b-a}{2}, \frac{d-c}{2}<b-c$.  However $b-c > \frac{d-c}{2}$ implies $d-b<\frac{d-c}{2}$ and $b-c >\frac{b-a}{2}$ implies $c-a <\frac{b-a}{2}$ by the pidgeonhole principle.  This means $m_1 < m_2$, a contradiction.
 
 Lastly, suppose that the support of one contains the other.  Without loss of generality, say $M\supseteq N$, so $a\leq c <d \leq b$.  Then,
 $$S_{M,N}=[c-a,d-a) \textrm{, and } S_{N,M}=[b-d,b-c). $$
 Since $m_1 \geq m_2$, 
 $$m_2=\tfrac{b-a}{2}\leq max\{c-a,b-d\}. $$
 If the maximum is $c-a$, then $c-a \geq \frac{b-a}{2} \geq b-c$ by the pidgeonhole principle.  Similarly, if the maximum is $b-d$, then $b-d \geq \frac{b-a}{2} \geq d-a$, as well.  Thus the result holds in all permissible cases.
 \end{proof}
We are now ready to prove our Main Result.
\begin{theorem}
\label{theorem main}
The previous list of examples is exhaustive.  Moreover, we have the following results:
\begin{enumerate}
\item[i.] $m_1 > m_2 \iff V^1$ is the origin $\iff$ the full progression is origin, axis, origin. 
\item[ii.] If $m_1=m_2 \iff V^1$ is an axis $\iff$ the full progression is axis, origin. 
\item[iii.] If $m_1 < m_2 \iff V^1$ is a hyperbola $\iff$ the full progression is hyperbola, plane, axis, origin or hyperbola, plane, origin.  
\end{enumerate}
\end{theorem}
Thus, in particular, $V^1$ can detect whether the interleaving distance $\epsilon$ comes from $m_1$ or $m_2$.  Note that $i$ corresponds to {\bf{Example \ref{origin1}}}, $ii.$ corresponds to {\bf{Example \ref{axis1}}}, and $iii.$ to {\bf{Examples \ref{hyperbola1}, \ref{hyperbola2}}}.
\begin{proof}
We will proceed in cases.

First, suppose $m_1 > m_2$.  Then,
\[ max\{|a-c|,|b-d|\} > max\left\{ \tfrac{b-a}{2},\tfrac{d-c}{2} \right\}, \]
so the interleaving distance $\epsilon$ is $m_2$.  Without loss of generality, assume $ \epsilon = \frac{b-a}{2}$, so we have 
$$\tfrac{d-c}{2} \leq \tfrac{b-a}{2} \textrm{, and either}  \tfrac{b-a}{2} < |a-c| 
\text{ or }  \tfrac{b-a}{2} < |b-d|.$$

At $\epsilon=\frac{b-a}{2}$, we'll first show that $V^1$ is the origin.  First, clearly at $\epsilon, H(M,M2\epsilon)\text{ and } H(N,N2\epsilon) $ are both $\{0\}$.  Suppose $\frac{b-a}{2} < |a-c|$.  Then, we need only show that $H(M,N\epsilon)=\{0\}=H(N,M\epsilon)$.
But, if $c>a$, then $H(M,N.\epsilon)=\{0\} \text{ since } c-\epsilon>a$.  Moreover, $H(N,M\epsilon)=\{0\} \text{ since } b-\epsilon<a $.  

On the other hand, if $c<a$, then $H(M,N\epsilon)=\{0\} \text{ since } d-\epsilon<a$ and $H(N,M\epsilon)=\{0\} \text{ since } a-\epsilon>c$.  Similarly, if $\frac{b-a}{2} < |b-d|$, then  $d>b$, so $H(M,N\epsilon)=\{0\} \text{ since } c-\epsilon>a $.  Also, 
$H(N,M\epsilon)=\{0\} \text{ since } b-\epsilon<c.$  Similarly, if $d<b$, then both homomorphsims are necessarily zero.

Thus, we have shown that at $\epsilon$ the variety $V^1$ is the origin.  We now continue within this case.  By {\bf{Proposition} \ref{prop hom}}, $m_1=max\{|a-c|,|b-d|\}$ is the value when the last homomorphism is born.  By {\bf{Lemma} \ref{lemma dead hom}}, the variety $V^2$ which must be an axis, appears at $m_1$.  But at $max\{|a-c|,|b-d|\}$, exactly one of $Hom(M,N\epsilon), Hom(M,N\epsilon)$ is nonzero, thus at $max\{|a-c|,|b-d|\}$ the variety is an axis.  By {\bf{Proposition} \ref{prop hom}}, the last variety $V^3$ is zero, which occurs at $max\{|a-d|,|b-c|\}=max\{\sigma',\tau'\}$.

Now, suppose that $m_1 = m_2 $, so $max\{|a-c|,|b-d|\} = max\left\{ \frac{b-a}{2},\frac{d-c}{2} \right\}$.  By our comments in {\bf{Lemma} \ref{lemma dead hom}} necessarily, the supports of $M, N$ overlap, and neither contains the other.  Without loss of generality, say $a < c<b<d$.  Then,
$$S_{M,N}=[m_1,max\{d-a,b-c\})=[m_1,d-a)\textrm{, and }S_{N,M}=[0,b-c). $$
First, say $c-a = m_1 = max\{c-a,d-b\}$, so
$$ c-a = \tfrac{b-a}{2} \textrm{, or } c-a = \tfrac{d-c}{2}.$$
If $c=a+\frac{b-a}{2} \geq a + \frac{d-c}{2}$, then
$$a \leq a < d-\tfrac{b-a}{2} \leq b\textrm{, so } H(M,N\epsilon) \neq \{0\}.$$
However, $H(N,M\epsilon)=\{0\}$, since
$$a-\tfrac{b-a}{2} \leq a+\tfrac{b-a}{2} =b-\tfrac{b-a}{2} \leq d. $$
If instead $c=a+\frac{d--c}{2}\geq a+\frac{b-a}{2}$, a similar analysis shows that 
$$H(M,N\epsilon)\neq \{0\}\textrm{, and }H(N,M\epsilon)=\{0\}. $$
Thus, at $\epsilon$ equal to the interleaving distance, one homomorphism is born, and clearly both 
$$H(M,M2\epsilon) =\{0\}\textrm{, and } H(N,N2\epsilon)=\{0\}.$$
Thus, we have that $V^1$ is given by an axis (in this case the $x$-axis).  By {\bf{Lemma} \ref{lemma dead hom}}, the last variety $V^2$ is the origin, which occurs at $max\{|a-d|,|b-c|\}=max\{\sigma',\tau'\}$, which is $d-b$ in our case.

Lastly, suppose $m_1 < m_2$.  Then, $max\{|a-c|,|b-d|\} < max\left\{ \frac{b-a}{2},\frac{d-c}{2} \right\}$, and the Interleaving distance is $\epsilon = max\{|a-c|,|b-d|\} $.  Clearly, at least one of 
$$H(M,M2\epsilon), H(N,N2\epsilon) \textrm{ will survive}. $$
First, suppose the supports of $M, N$ intersect, but neither contains the other.  Without loss of generality, say $a<c<b<d$.  Then,
$$S_{M,N}=[m_1,d-a), S_{N,M}=[0,b-c). $$
We will show that both homomorphisms are alive at the interleaving distance $\epsilon=m_1$.    But this is clear, since 
\begin{eqnarray*}
d-b&\leq& c-a<\tfrac{b-a}{2} \implies b-c > \tfrac{b-a}{2} >c-a \textrm{, and}\\
d-b&\leq& c-a<\tfrac{d-c}{2} \implies b-c > \tfrac{d-c}{2} >c-a.
\end{eqnarray*}
Similarly, 
\begin{eqnarray*}
c-a&\leq& d-b<\tfrac{d-c}{2} \implies b-c > \tfrac{d-c}{2} >d-b \textrm{, and}\\
c-a&\leq& d-b<\tfrac{b-a}{2} \implies b-c > \tfrac{b-a}{2} >d-b.\\
\end{eqnarray*}
Therefore, we have established that in this case, the first variety $V^1$ is a hyperbola.  Moreover, the above calculatiion shows that $m_2 < b-c$, thus $V^2$ is a plane begining at $\epsilon=m_2$.  Continuing, since $b-a <d-a$, $V^3$ is an axis ocurring at $\epsilon=b-a$ and $V^4$ is the origin, which starts at $\epsilon = d-a$.

Since disjoint supports are not possible when $m_1 < m_2$, it remains only to consider when, say $M \supseteq N$, so suppose $a \leq c < d \leq b$.  Then,
$$S_{M,N}=[c-a,d-a)\textrm{, and }S_{N,M}=[b-d,b-c). $$
Thus, by inspection, at the interleaving distance $m_1=max\{c-a.b-d\}$, both homomorphims are alive.  Thus, since $H(M,M2m_1)\neq \{0\}$, the first variety $V_1$ is a hyperbola.  Continuing, it follows from properties of interleavings that $m_2 =\frac{b-a}{2} <max\{d-a,b-c\}$, but we one easily checks that
$$\tfrac{b-a}{2}\geq d-a \implies b-d\geq \tfrac{b-a}{2}\textrm{, and }\tfrac{b-a}{2}\geq b-c \implies c-a \geq \tfrac{b-a}{2}. $$
Thus, $V^2$ is the plane, which occurs at $m_2=\tfrac{b-a}{2}$.  Now let
$$r=min\{d-a,b-c\} \textrm{, and }s=max\{d-a,b-c\}. $$
If $r=s$, then both homomorhisms die at the same time, and the progression ends with $V^3$ equal to the origin.  Alternatively, if $r<s$, then $V^3$ is an axis ocurring at $r$, and $V^4$ is the origin beginning at $s$.

Since the cases, 
$$m_1 >m_2, m_1=m_2, m_1<m_2$$
constitute a partition, the result holds.
\end{proof}

\bibliography{master_bib}
\bibliographystyle{abbrv}

\end{document}